\documentclass{compositio}
\usepackage{amsmath, amsthm, amssymb, stmaryrd}
\usepackage{mathrsfs,array}
\usepackage{xy}
\usepackage{hyperref}

\input xy
\xyoption{all}

\newcommand{\QQ}{{\overline{\bf{Q}}}}
\newcommand{\C}{{\bf{C}}}

\newcommand{\OO}{{\mathcal O}}

\newcommand{\Fc}{\mathscr{F}}

\newcommand{\from}{{\colon}}
\newcommand{\injects}{\hookrightarrow}

\newcommand{\isom}{\cong}
\newcommand{\eps}{\varepsilon}

\DeclareMathOperator{\ch}{char}
\DeclareMathOperator{\LT}{LT}

\DeclareMathOperator{\Spf}{Spf}
\DeclareMathOperator{\ts}{ts}

\DeclareMathOperator{\End}{End} \DeclareMathOperator{\Hom}{Hom}
\DeclareMathOperator{\Gal}{Gal}
\DeclareMathOperator{\GL}{GL}\DeclareMathOperator{\Spec}{Spec}
\DeclareMathOperator{\Aut}{Aut}

\DeclareMathOperator{\tr}{Tr}

\DeclareMathOperator{\Ind}{Ind}
\DeclareMathOperator{\SL}{SL}

\numberwithin{equation}{subsection}
\newtheorem{Theorem}{Theorem}[section]
\newtheorem{lemma}[Theorem]{Lemma}

\newtheorem{prop}[Theorem]{Proposition}
\newtheorem{conj}[Theorem]{Conjecture}

\theoremstyle{definition}
\newtheorem{defn}[Theorem]{Definition}

\newtheorem{rmk}[Theorem]{Remark}

\newcommand{\abs}[1]{\left\lvert #1 \right\rvert}

\newcommand{\set}[1]{\left\{ #1 \right\}}

\newcommand{\tbt}[4]{\left(\begin{matrix}#1 & #2\\#3 & #4\end{matrix}\right)}

\newcommand{\powerseries}[1]{\llbracket #1 \rrbracket}
\newcommand{\tatealgebra}[1]{\langle #1 \rangle}
\newcommand{\ps}{\powerseries}
\newcommand{\ta}{\tatealgebra}

\newcommand{\nr}{{\text{nr}}}

\newcommand{\F}{{\mathbf{F}}}
\newcommand{\FF}{{\overline{\mathbf{F}}}}
\newcommand{\gp}{{\mathfrak{p}}}

\DeclareMathOperator{\Frob}{Frob}

\DeclareMathOperator{\Spm}{Spm}

\DeclareMathOperator{\N}{N}
\DeclareMathOperator{\univ}{univ}

\DeclareMathOperator{\sgn}{sgn}

\title{Good reduction of affinoids on the Lubin-Tate tower}
\author{Jared Weinstein}
\thanks{The author was supported by NSF Postdoctoral Fellowship DMS-0803089.}
\email{jared@math.ucla.edu}
\address{Dept. of Mathematics\\ University of California, Los Angeles\\ Box 951555\\ Los Angeles, CA 90095-1555}

\begin{document}

\begin{abstract}
We analyze the geometry of the tower of Lubin-Tate deformation spaces,
which parametrize deformations of a one-dimensional formal module of
height $h$ together with level structure.  According to the conjecture
of Deligne-Carayol, these spaces realize the local Langlands
correspondence in their $\ell$-adic cohomology.  This conjecture is
now a theorem, but currently there is no purely local proof.  Working
in the equal characteristic case, we find a family of affinoids in the
Lubin-Tate tower with good reduction equal to a rather curious
nonsingular hypersurface, whose equation we present explicitly.
Granting a conjecture on the $L$-functions of this hypersurface, we
find a link between the conjecture of Deligne-Carayol and the theory
of Bushnell-Kutzko types, at least for certain class of wildly
ramified supercuspidal representations of small conductor.
\end{abstract}

\maketitle

\section{Introduction}

Let $F$ be a non-archimedean local field.    By the local Langlands correspondence, the irreducible admissible representations of $\GL_h(F)$ are parametrized in a systematic way by $h$-dimensional representations of the Weil-Deligne group of $F$.    This is established in~\cite{LRS} for fields of positive characteristic and in ~\cite{HenniartPreuveSimple} and~\cite{HarrisTaylor:LLC} for $p$-adic fields.   The local Langlands correspondence appears in a geometric context;  namely it is realized in the cohomology of the ``Lubin-Tate tower", a projective system of deformation spaces of a one-dimensional formal $\OO_F$-module of height $h$, cf.~\cite{DrinfeldEllipticModules}.  We refer to this phenomenon as the conjecture of Deligne-Carayol, after the paper ~\cite{Carayol:nonabelianLT} which contains the precise statement of the conjecture.  The papers~\cite{Carayol:ladicreps} and~\cite{Carayol:Mauvaise} prove the conjecture for the case $h=2$.   The complete conjecture of Deligne-Carayol was proved in~\cite{Boyer} for fields of positive characteristic and in~\cite{HarrisTaylor:LLC} for $p$-adic fields.  Both papers involve embedding $F$ into a global field and appealing to results from the theory of Shimura varieties or Drinfeld modular varieties.

In~\cite{Harris}, Harris identifies some unsettled problems in the study of the local Langlands correspondence, and top among these is the lack of a purely local proof of the correspondence.     Bushnell and Kutzko's theory of types~\cite{BushnellKutzko} parametrizes admissible representations of $\GL_h(F)$ by finite-dimensional characters of open compact-mod-center subgroups.  Naturally one hopes to link the parametrization by types to the parametrization by Weil-Deligne representations, so that one might obtain an ``explicit local Langlands correspondence."  There have been some remarkable efforts in this direction, see~\cite{HenniartLanglandsKazhdan}, \cite{BushnellHenniart:I},~\cite{BushnellHenniart:II},~\cite{Henniart:Bushnell}, but these do not seem to interface with the geometric interpretation of the local Langlands correspondence afforded by the conjecture of Deligne-Carayol.   Harris asks (\cite{Harris}, Question 9) whether the Bushnell-Kutzko types can be realized in the cohomology of analytic subspaces of the Lubin-Tate tower.

In the present effort we demonstrate progress towards an affirmative answer to this question.   We construct a family of open affinoids $\mathfrak{Z}$ of the Lubin-Tate tower which have good reduction equal to a hypersurface $\overline{\mathfrak{Z}}$ whose equation we give explicitly, cf. Thm.~\ref{mainthm} below.  The cohomology of these affinoids appears to contain exactly the Bushnell-Kutzko types for those supercuspidal representations whose Weil parameters are of the form $\Ind_{E/F}\theta$, where $E/F$ is the unramified extension of degree $h$ and $\theta$ is a character of the Weil group of $E$ of conductor $\gp_E^2$, where $\gp_E$ is the maximal ideal of $\OO_E$.  We refer to these as the unramified supercuspidals of level $\pi^2$.  The action of the Weil group on $\overline{\mathfrak{Z}}$ is completely transparent.  The question of whether the affinoids $\mathfrak{Z}$ really do realize the local Langlands correspondence for such representations is reduced to the calculation of certain $L$-functions attached to $\overline{\mathfrak{Z}}$, see Conj.~\ref{mainconj}.

It is hoped that this paper will initiate a systematic study of open affinoids with good reduction in the Lubin-Tate tower.   The best outcome would be the construction of a semistable model for the Lubin-Tate spaces, using an appropriate covering by open affinoids.   This is precisely what is done in~\cite{coleman:Xp3} for the classical modular curves $X_0(Np^3)$, and in~\cite{WeinsteinStableModels} for Lubin-Tate curves with arbitrary level structure.  Then the weight spectral sequence of Rapoport-Zink~\cite{RZ} would compute the cohomology of the Lubin-Tate tower in terms of the reduction of the semistable model.  A purely local proof of the conjecture of Deligne-Carayol would then be reduced to the computation of the zeta functions associated to the components of the reduction of the semistable model.

Before stating our main theorem, we introduce some notation.    We write $\mathfrak{X}(\pi^n)$, $n\geq 0$, for the system of rigid-analytic spaces comprising the Lubin-Tate tower of deformations of a height $h$ one-dimensional formal $\OO_F$-module with Drinfeld level $\pi^n$ structure;  see \S\ref{definitions} for definitions.  Crucial to the analysis are the ``canonical points" of $\mathfrak{X}(\pi^n)$ arising from the canonical liftings of Gross~\cite{Gross:canonical}:  these are the deformations with extra endomorphisms by the ring of integers in a separable extension $E/F$.  Such a point is defined over the extension $E_n/\hat{E}^{\nr}$ obtained by adjoining the $\pi^n$-division points of a formal Lubin-Tate $\OO_E$-module of height one.

In our analysis we concentrate on those canonical points for which the associated extension $E/F$ is unramified.  We refer to these as {\em unramified canonical points}.  By performing explicit computations with coordinates, we find certain affinoid neighborhoods around each unramified canonical point $x$ which have good reduction.  These neighborhoods lie in a space intermediate in the covering $\mathfrak{X}(\pi^2)\to\mathfrak{X}(\pi)$, which we call $\mathfrak{X}(K_{x,2})=\mathfrak{X}(\pi^2)/K_{x,2}$;  for details, see~\S\ref{invcoord}. Briefly put, $x$ determines an embedding of $\OO_F$-algebras $\OO_E\injects M_n(\OO_F)$, and $K_{x,2}$ is the congruence subgroup defined by
\[ K_{x,2}= \set{g\in 1+\pi M_n(\OO_F)\biggm\vert \tr((g-1)\OO_E)\subset\gp_F^2}. \]

Our main result is:
\begin{Theorem}
\label{mainthm}
Assume that $F$ has positive characteristic, with residue field $\F_q$.   Let $x\in\mathfrak{X}(\pi^2)$ be an unramified canonical point.  There exists an open affinoid neighborhood $\mathfrak{Z}$ of the image of $x$ in $\mathfrak{X}(K_{x,2})$ whose reduction is the smooth hypersurface $\overline{\mathfrak{Z}}$ in the variables $V_1,\dots,V_h$ defined by the equation
\[ \det\left(\begin{matrix}
V_1^{q^h}-V_1 & V_2^{q^h}-V_2 & V_3^{q^h}-V_3 & \cdots & V_{h-1}^{q^h} - V_{h-1} & V_h^{q^h}-V_h
\\
1 & V_1^q & V_2^q & \cdots & V_{h-2}^q & V_{h-1}^q
\\
0 & 1 & V_1^{q^2} & \cdots & V_{h-3}^{q^2} & V_{h-2}^{q^2}
\\
\vdots & & & \ddots & & \vdots
\\
0 & 0 & 0 & \cdots& 1 & V_1^{q^{h-1}} \end{matrix}\right)=0.\]
\end{Theorem}

\begin{rmk}\label{hypersurfaceremark}  Let $R$ be the noncommutative polynomial ring $\mathbf{F}_{q^h}[\tau]/(\tau^{h+1})$, whose multiplication law is given by $\tau\alpha=\alpha^q\tau$, $\alpha\in \mathbf{F}_{q^h}$.  Let $A=R\otimes_{\mathbf{F}_{q^h}} \mathbf{F}_{q^h}[V_1,\dots,V_h]$, and let $\Phi\from A\to A$ be the $R$-linear endomorphism which sends $V_i$ to $V_i^q$.
Let $g=1+V_1\tau+\dots+V_h\tau^h\in A^\times$;  then the coefficient of $\tau^n$ in $\Phi^{h}(g)g^{-1}$ is the determinant appearing in Thm.~\ref{mainthm}.  This shows that the hypersurface $\overline{\mathfrak{Z}}$ admits a large group of automorphisms, namely $R^\times$.  See \S\ref{Bcross} for an interpretation of this automorphism group in terms of the Jacquet-Langlands correspondence.
\end{rmk}

\begin{rmk}  We expect the condition $\ch F>0$ to be unnecessary.  This condition enables us to write down explicit models for universal deformations of formal $\OO_F$-modules with level structure, as in \S\ref{univdefcharp}.  It may be possible to remove this condition if one is more careful with error terms.
\end{rmk}

\begin{rmk} In Yoshida's paper~\cite{yoshida} the space $\mathfrak{X}(\pi)$ is treated, with no condition on the characteristic of $F$.   In that case one finds an affinoid subdomain of $\mathfrak{X}(\pi)\otimes E_1$ whose reduction is the Deligne-Lusztig variety for $\GL_h(k)$, see \S\ref{tame}.   Based on this calculation, Yoshida proceeds to show that the vanishing cycles of $\mathfrak{X}(\pi)$ realize the local Langlands correspondence for supercuspidal representations ``of depth zero".
\end{rmk}

\begin{rmk}  Thm.~\ref{mainthm} agrees well with our work in~\cite{WeinsteinStableModels}, which gives a detailed description of a stable reduction of the tower $\mathfrak{X}(\pi^n)$ when $h=2$.  In this case the curve $\overline{\mathfrak{Z}}$ is isomorphic over $\mathbf{F}_q$ to a disjoint union of copies of the ``Hermitian curve" $Y+Y^q=V^{q+1}$.   The Hermitian curve also happens to be isomorphic over $\overline{\mathbf{F}}_q$ to the Deligne-Lusztig curve for $\SL_2(\F_q)$, but this seems to be a coincidence which does not persist for $h>2$.
\end{rmk}

In order to apply Thm.~\ref{mainthm} to the conjecture of Deligne-Carayol, it will be necessary to calculate the compactly supported $\ell$-adic cohomology of $\overline{\mathfrak{Z}}$, $\ell\neq p$, as a module for the action of the stabilizer of $\mathfrak{Z}$ in $\GL_2(\OO_F)$, which is the group $U^1=1+\pi M_h(\OO_F)$.  This in turn is equivalent to the calculation of the $L$-functions of some $\ell$-adic sheaves on affine $(h-1)$-space.  To wit, let $X$ be the hypersurface over $\F_{q^h}$ whose equation is the one appearing in Thm.~\ref{mainthm}.  Then $X$ is an Artin-Schreier cover of $\mathbf{A}^{h-1}/\F_{q^h}$ with Galois group $\F_{q^h}$.  For each character $\psi$ of $\F_{q^h}$ with values in $\QQ_\ell^\times$, let $\mathscr{L}_\psi$ be the corresponding lisse rank one sheaf on $\mathbf{A}^{h-1}$.  Then the zeta function $Z(X,t)$ factors as a product of the $L$-functions $L(\mathbf{A}^{h-1},\mathscr{L}_\psi,t)$ as $\psi$ runs over characters of $\F_{q^h}$.
\begin{conj}
\label{mainconj}
Suppose $\psi$ does not factor through $\tr_{\F_{q^h}/\F_{q^d}}$ for any proper divisor $d$ of $h$.  Then
\[
L(\mathbf{A}^{h-1},\mathscr{L}_\psi,t) = \left(1+(-1)^hq^{\frac{h(h-1)}{2}}t\right)^{(-1)^hq^{\frac{h(h-1)}{2}}}.
\]
\end{conj}

The formula in Conj.~\ref{mainconj} is striking:  it implies that the contribution of the $\psi$-part of the Euler characteristic $H^*_c(X\otimes\overline{\F}_q,\QQ_\ell)$ to the quantities $\#X(\F_{q^h}),\#X(\F_{q^{2h}}),\dots$ is the maximum possible under the constraints of the Riemann hypothesis for $X$.  In fact we strongly suspect that $X$ has the maximum number of $\F_{q^{hn}}$-rational points relative to its compactly supported Betti numbers.    More to the point,  Conj.~\ref{mainconj} would also imply that $H_c^{h-1}(\mathfrak{Z},\QQ_\ell)$ realizes the Bushnell-Kutzko types for the unramified supercuspidals of level $\pi^2$, and that the action of the Weil group of $F$ on $\mathfrak{Z}$ is in accord with the local Langlands correspondence.  We postpone the details of this claim for future work, but see~\cite{WeinsteinENLTT}, \S 4 and \S 5 for a comprehensive calculation in the case $h=2$.

Conj.~\ref{mainconj} itself can be verified quite easily for $h=2$, in which case $X$ is a disjoint union of $q$ copies of the Hermitian curve $Y^q+Y=X^{q+1}$:  this curve is ``maximal" over $\F_{q^2}$ in the sense that it attains the Hasse-Weil bound for the maximum number of $\F_{q^2}$-rational points.  Conj.~\ref{mainconj} can be verified numerically for small values of $q$ and $h>2$, but unfortunately we cannot give a general proof at this time.  The polynomial on the right-hand side of the equation in Thm.~\ref{mainthm} is degenerate in the sense of~\cite{AdolphsonSperber}, which frustrates efforts to determine even the degree of the rational function $L(\mathbf{A}^{h-1},\mathscr{L}_\psi,t)$.

The construction of the explicit local Langlands correspondence for unramified supercuspidals appears in~\cite{HenniartLanglandsKazhdan}.  A salient feature of that paper is the discrepancy between two means of passing from a regular character of $E^\times$ to a supercuspidal representation of $\GL_h(F)$.  The first construction is the local Langlands correspondence, the second construction is induction from a compact-mod-center subgroup, and the discrepancy, which appears exactly when $h$ is even, manifests as the nontrivial unramified quadratic character of $E^\times$.   Granting Conj.~\ref{mainconj}, we arrive at a geometric explanation for this behavior in terms of the eigenvalue of Frobenius on the middle cohomology of the hypersurface $X$, for these are positive if and only if $h$ is odd.   In the subsequent papers~\cite{BushnellHenniart:I} and~\cite{BushnellHenniart:II} on the explicit local Langlands correspondence there is a systematic treatment of this discrepancy between the two constructions in the ``essentially tame" case;  we find it very likely that this discrepancy can always be explained by the behavior of Frobenius eigenvalues acting on the cohomology of an open affinoid in the Lubin-Tate tower having good reduction.

We outline our work:  In \S2, we review the relevant background material from~\cite{DrinfeldEllipticModules} on one-dimensional formal modules and the Lubin-Tate tower.  In \S3, we impose the condition that $\ch F>0$ and establish a functorial construction of top exterior powers of one-dimensional formal $\OO_F$-modules which may be of independent interest.  The heart of the paper is \S4.  Given an unramified canonical point $x$ in $\mathfrak{X}(\pi^2)$, we construct a coordinate $Y$ on that space which is invariant under $K_{x,2}$.  The coordinate $Y$ is integral on a certain affinoid neighborhood of $x$ in $\mathfrak{X}(\pi^2)$, and the reduction of the minimal polynomial for $Y$ over the ring of integral functions on $\mathfrak{X}(1)$ gives the equation appearing in Thm.~\ref{mainthm}.  We conclude in \S5 with some basic observations about the hypersurface $\overline{\mathfrak{Z}}$ which we hope will illuminate the formulas in Conj.~\ref{mainconj} and motivate future work linking Thm.~\ref{mainthm} to the local Langlands and Jacquet-Langlands correspondences for $\GL_h(F)$.

\section{Preliminaries on formal modules}

\subsection{Definitions}
\label{definitions}
Throughout this paper, $F$ is a local non-archimedean field with ring of integers $\OO_F$, uniformizer $\pi$ and residue field $k$ having cardinality $q$, a power of the prime $p$.  Let $\gp$ be the maximal ideal of $\OO_F$, and let $v$ be the valuation on $F$, normalized so that $v(\pi)=1$.   We also use $v$ for the unique extension of this valuation to finitely ramified extension fields $E$ of $F$ contained in the completion of the separable closure of $F$.

\begin{defn}  Let $R$ be a commutative $\OO_F$-algebra, with structure map $i\from \OO_F\to R$.  A {\em formal one-dimensional $\OO_F$-module} over $R$ is a power series $\Fc(X,Y)=X+Y+\dots\in R\ps{X,Y}$ which is commutative, associative, admits 0 as an identity, together with a power series $[a]_\Fc(X)\in R\ps{X}$ for each $a\in \OO_F$ satisfying $[a]_\Fc(X)\equiv i(a)X\pmod{X^2}$ and $\Fc([a]_\Fc(X), [a]_\Fc(Y))=[a]_\Fc(\Fc(X,Y))$.
\end{defn}

The addition law on a formal $\OO_F$-module $\Fc$ will usually be written $X+_\Fc Y$.
If $\Fc$ and $\Fc'$ are two formal $\OO_F$-modules, there is an evident notion of an  isogeny $\Fc\to \Fc'$, and $\Hom(\Fc,\Fc')$ has the structure of an $\OO_F$-module.

If $R$ is a $k$-algebra, we either have $[\pi]_\Fc(X)=0$ or else $[\pi]_\Fc(X)=f(X^{q^h})$ for some power series $f(X)$ with $f'(0)\neq 0$.  In the latter case, we say $\Fc$ has height $h$ over $R$.

Fix an integer $h\geq 1$.  Let $\Sigma$ be a one-dimensional formal $\OO_F$-module over $\overline{k}$ of height $h$.  The functor of deformations of $\Sigma$ to complete local Noetherian $\hat{\OO}_{F^{\nr}}$-algebras is representable by a universal deformation $\Fc^{\univ}$ over an algebra $\mathcal{A}$ which is isomorphic to the power series ring $\hat{\OO}_{F^{\nr}}\ps{u_1,\dots,u_{h-1}}$ in $(h-1)$ variables, cf.~\cite{DrinfeldEllipticModules}.   That is, if $A$ is a complete local $\hat{\OO}_F^{\nr}$-algebra with maximal ideal $P$, then the isomorphism classes of deformations of $\Sigma$ to $A$ are given exactly by specializing each $u_i$ to an element of $P$ in $\Fc^{\univ}$.

\subsection{The universal deformation in the positive characteristic case}
\label{univdefcharp}   The results of the previous paragraph take a very simple form in the equal characteristic case.  Assume $\ch F=p$, so that $F=k((\pi))$ is the field of Laurent series over $k$ in one variable, with $\OO_F=k\ps{\pi}$.  Then a model for $\Sigma$ is given by the simple rules
\begin{align*}
X+_\Sigma Y &= X+Y\\
[\zeta]_{\Sigma}(X)&=\zeta X,\;\zeta\in k \\
[\pi]_{\Sigma}(X)&= X^{q^h}
\end{align*}
The universal deformation $\Fc^{\univ}$ also has a simple model over $\mathcal{A}$:
\begin{align}
X+_{\Fc^{\univ}} Y&=X+Y \nonumber\\
[\zeta]_{\Fc^{\univ}}(X)&=\zeta X,\;\zeta\in k \nonumber\\
\label{piuniv}
[\pi]_{\Fc^{\univ}}(X)&= \pi X + u_1X^q + \dots+  u_{h-1}X^{q^{h-1}}+X^{q^h}.
\end{align}

Let $\OO_B=\End\Sigma$, and let $B=\OO_B\otimes_{\OO_F} F$.  Then $B$ is the central division algebra over $F$ of invariant $1/h$.  Let $k_h/k$ be the field extension of degree $h$:  then $\OO_B$ is generated by the unramified extension $\OO_{E}=k_h\ps{\pi}$ of $\OO_K$ of degree $h$, which acts on $\Sigma$ in an evident way, together with the endomorphism $\Phi(X)=X^q$.  (The relations are $\Phi^h=\pi$ and $\Phi \zeta=\zeta^q\Phi$, $\zeta\in k_h$.)   Inasmuch as $\mathcal{A}=\hat{\OO}_F^\nr\ps{u_1,\dots,u_{h-1}}$ is the moduli space of deformations of $\Sigma$, the automorphism group $\Aut\Sigma=\OO_B^\times$ acts naturally on $\mathcal{A}$.    It is natural to ask how $\OO_B^\times$ acts on the level of coordinates.  The action of an element $\zeta\in k_n^\times$ is simple enough: $\zeta(u_i) = \zeta^{q^i-1} u_i$, $i=1,\dots,h-1$.  On the other hand the action of an element such as $1+\Phi\in \OO_B^\times$ seems difficult to give explicitly.

\subsection{Moduli of deformations with level structure}  Let $A$ be a complete
local $\OO_F$-algebra with maximal ideal $M$, and let $\Fc$ be a one-dimensional formal $\OO_F$-module over $A$, and let $h\geq 1$ be the height of $\Fc\otimes A/M$.
\begin{defn} Let $n\geq 1$.  A {\em Drinfeld level $\pi^n$ structure} on $\Fc$ is an $\OO_F$-module homomorphism $\phi\from (\pi^{-n}\OO_F/\OO_F)^{\oplus h}\to M$ for which the relation \[\prod_{x\in (\gp^{-1}/\OO_F)^{\oplus h}}\left(X-\phi(x)\right)\biggm\vert [\pi]_{\Fc}(X)\] holds in $A\ps{X}$.  If $\phi$ is a Drinfeld level $\pi^n$ structure, the images under $\phi$ of the standard basis elements $(\pi^{-n},0,\dots,0),\dots,(0,0,\dots,\pi^{-n})$ of $(\gp^{-n}/\OO_F)^{\oplus h}$ form a {\em Drinfeld basis} of $\Fc[\pi^n]$.
\end{defn}
%

Fix a formal $\OO_F$-module $\Sigma$ of height $h$ over $\overline{k}$.   Let $A$ be a noetherian local $\hat{\OO}_F^{\nr}$-algebra such that the structure morphism $\hat{\OO}_F^{\nr}\to A$ induces an isomorphism between residue fields.  A {\em deformation} of $\Sigma$ with level $\pi^n$ structure over $A$ is a triple $(\Fc,\iota,\phi)$, where $\iota\from \Fc\otimes \overline{k} \to \Sigma$ is an isomorphism of $\OO_F$-modules over $\overline{k}$ and $\phi$ is a Drinfeld level $\pi^n$ structure on $\Fc$.

\begin{prop}{\cite{DrinfeldEllipticModules}}
The functor which assigns to each $A$ as above the set of deformations of $\Sigma$ with Drinfeld level $\pi^n$ structure over $A$ is representable by a regular local ring $\mathcal{A}(\pi^n)$ of relative dimension $h-1$ over $\hat{\OO}_F^\nr$.  Let $X_1^{(n)},\dots,X_h^{(n)}\in\mathcal{A}(\pi^n)$ be the corresponding Drinfeld basis for $\Fc^{\univ}[\pi^n]$;  then these elements form a set of regular parameters for $\mathcal{A}(\pi^n)$.
\end{prop}

There is a finite injection of $\hat{\OO}_F^{\nr}$-algebras $\mathcal{A}(\pi^n)\to\mathcal{A}(\pi^{n+1})$ corresponding to the obvious degeneration map of functors.  We therefore may consider $\mathcal{A}(\pi^n)$ as a subalgebra of $\mathcal{A}(\pi^{n+1})$, with the equation $[\pi]_u\left(X_i^{(n+1)}\right)=X_i^{(n)}$ holding in $\mathcal{A}(\pi^{n+1})$.

Let $X(\pi^n)=\Spf \mathcal{A}(\pi^n)$, so that $X(\pi^n)$ is a formal scheme of relative dimension $h-1$ over $\Spf \hat{\OO}^{\nr}_F$.  Let $\mathfrak{X}(\pi^n)$ be the generic fiber of $X(\pi^n)$;  then $\mathfrak{X}(\pi^n)$ is a rigid analytic variety.  The coordinates $X_i^{(n)}$ are then analytic functions on $\mathfrak{X}(\pi^n)$ with values in the open unit disc.   We have that $\mathfrak{X}(1)$ is the rigid-analytic open unit polydisc of dimension $h-1$.

The group $\GL_h(\OO_F/\pi^n\OO_F)$ acts on the right on $\mathfrak{X}(\pi^n)$ and on the left on $\mathcal{A}(\pi^n)$.  The degeneration map $\mathfrak{X}(\pi^n)\to\mathfrak{X}(1)$ is Galois with group $\GL_h(\OO_F/\pi^n\OO_F)$.  For an element $M\in\GL_h(\OO_F/\pi^n\OO_F)$ and an analytic function $f$ on $\mathfrak{X}(\pi^n)$, we write $M(f)$ for the translated function $z\mapsto f(zM)$.  When $f$ happens to be one of the parameters $X_i^{(n)}$, there is a natural definition of $M\left(X_i^{(n)}\right)$ when $M\in M_h(\OO_F/\pi^n\OO_F)$ is an arbitrary matrix: if $M=(a_{ij})$, then
\begin{equation}
\label{MX}
M\left(X_i^{(n)}\right)=[a_{i1}]_{\Fc^{\univ}}\left(X_1^{(n)}\right)+_{\Fc^{\univ}}\dots+_{\Fc^{\univ}}[a_{ih}]_{\Fc^{\univ}}\left(X_h^{(n)}\right).
\end{equation}

\section{Determinants}
\label{determinants}
A natural first question in the study of the Lubin-Tate tower $\mathfrak{X}(\pi^n)$ is to compute its zeroth cohomology;  {\em i.e.} to determine its geometrically connected components along with the appropriate group actions.  This question is answered completely by Strauch in~\cite{Strauch:ConnComp}.   Let $\LT$ be a one-dimensional formal $\OO_F$-module over $\hat{\OO}_{F^{\nr}}$ for which $\LT\otimes \overline{k}$ has height one.  Let $F_0=\hat{F}^{\nr}$, and for $n\geq 1$, let $F_n=F_0(\LT[\pi^n])$ be the classical Lubin-Tate extension.  Let $\chi\from\Gal(F_n/F_0)\to (\OO_F/\pi^n\OO_F)^\times$ be the isomorphism of local class field theory, so that $\Gal(F_n/F_0)$ acts on $\LT[\pi^n]$ through $\chi$.   Finally, let $\mathfrak{X}_{\LT}(\pi^n)$ be the (zero-dimensional) space of deformations of $\LT\otimes\overline{k}$ with Drinfeld $\pi^n$ structure, so that $\mathfrak{X}_{\LT}(\pi^n)(F_n)$ is the set of bases for $\LT[\pi^n](F_n)$ as a free $(\OO_F/\pi^n \OO_F)$-module of rank one.   We now paraphrase~\cite{Strauch:ConnComp}, Thm. 4.4 in the context of the rigid-analytic spaces $\mathfrak{X}(\pi^n)$.
\begin{Theorem}
\label{components}
The geometrically connected components of $\mathfrak{X}(\pi^n)$ are defined over $F_n$, and there is a bijection \[  \pi_0(\mathfrak{X}(\pi^n)\otimes F_n)\tilde{\longrightarrow} \mathfrak{X}_{\LT}(\pi^n)(F_n).\]  Under this bijection, the action of an element $(g,b,\tau)$ in $\GL_h(\OO_F)\times\OO_B^\times\times \Gal(F_n/F_0)$ on $\mathfrak{X}_{\LT}(\pi^n)(F_n)$ is through the character
\begin{equation}
\label{tripleproductchar} (g,b,\tau)\mapsto \det(g)\N_{B/F}(b)^{-1}\chi(\tau)^{-1} \in (\OO_F/\pi^n\OO_F)^\times.
\end{equation}
\end{Theorem}
(In~\cite{Strauch:ConnComp}, $\pi_0(\mathfrak{X}(\pi^n)\otimes\C_{\pi})$ is identified with $\pi_0(\Spec(F_n\otimes_{F_0} \C_\pi))$, where $\C_\pi$ is the completion of a separable closure of $F$.  But this latter $\pi_0$, being the set of $F_0$-linear embeddings of $F_n$ into $\C_\pi$, is the same as the set of bases for $\LT[\pi^n](\C_\pi)$.   Thus Thm.~\ref{components} carries the same content as the theorem cited in~\cite{Strauch:ConnComp}.)

As noted in the introduction to~\cite{Strauch:ConnComp}, Thm.~\ref{components} suggests a determinant functor $\Fc\mapsto \Lambda^h\Fc$ assigning to each deformation $\Fc$ of $\Sigma$ a deformation $\Lambda^h\Fc$ of $\LT\otimes \overline{k}$.  This functor would of course identify the top exterior power of the Tate module $T(\Fc)$ with $T(\Lambda^h \Fc)$.   In this section we provide just such a determinant functor {\em in the case of equal characteristic}, taking advantage of the explicit model of the universal deformation $\Fc^{\univ}$ described in \S\ref{univdefcharp}.  More precisely we prove:
\begin{Theorem}   \label{functorial}  Assume $\ch F>0$.  For each $n\geq 1$ there exists a morphism
\[ \mu_n \from \Fc^{\text{univ}}[\pi^n]\times\cdots\times \Fc^{\text{univ}}[\pi^n] \to\LT[\pi^n]\otimes\mathcal{A} \]
of group schemes over $\mathcal{A}=\hat{\OO}_{F^{\nr}}\ps{u_1,\dots,u_{h-1}}$ which is $\OO_F$-multilinear and alternating, and which satisfies the following properties:
\begin{enumerate}
\item  The maps $\mu_n$ are compatible in the sense that
\[\mu_n([\pi]_{\Fc^{\univ}}(X_1),\dots,[\pi]_{\Fc^{\univ}}(X_h))=\mu_{n-1}(X_1,\dots,X_h)\] for $n\geq 2$.
\item If $X_1,\dots,X_h$ are sections of $\Fc^{\univ}[\pi^n]$ over an $\mathcal{A}$-algebra $R$ which form a Drinfeld level $\pi^n$ structure, then $\mu_n(X_1,\dots,X_h)$ is a Drinfeld level $\pi^n$ structure for $\LT[\pi^n]\otimes R$.
\end{enumerate}
\end{Theorem}

\begin{rmk}  It is also possible to show that $\mu_n$ transforms the action of $\GL_h(\OO_F)\times \OO_B^\times \times\Gal(F_n/\hat{F}^{\nr})$ on $\Fc^{\text{univ}}[\pi^n]\times\cdots\times \Fc^{\text{univ}}[\pi^n]$ into the character defined in Eq.~\eqref{tripleproductchar}, but we will not be needing this.
\end{rmk}

The proof of Thm.~\ref{functorial} will occupy \S\ref{detlevelpi} and \S\ref{dethigherlevel}.  Up to isomorphism there is only one formal $\OO_F$-module $\LT$ whose reduction has height one, so we are free to choose a model for it.
For the remainder of the paper, $\LT$ will denote the formal $\OO_F$-module over $\hat{\OO}_{F^{\nr}}$ with operations
\begin{align*}
X+_{\LT} Y &= X+Y \\
[\alpha]_{\LT}(X) &= \alpha X,\; \alpha\in k\\
[\pi]_{\LT}(X) &= \pi X + (-1)^{h-1} X^q.
\end{align*}

\subsection{\texorpdfstring{Determinants of level $\pi$ structures}{Determinants of level pi structures}}
\label{detlevelpi}

First define the polynomial in $h$ variables
\[ \mu(X_1,\dots,X_h) = \det\left(X_i^{q^j}\right)\in k[X_1,\dots,X_h]\]
(the exponent $j$ ranges from 0 to $h-1$).  Then $\mu$ is a $k$-linear alternating form, known as the Moore determinant, cf.~\cite{Goss}, Ch. 1.  We will need two simple identities involving $\mu$.  The first is
\begin{equation}
\label{prodmu}
\prod_{0\neq a\in k^h}(a_1X_1+\dots+a_hX_h) = (-1)^h \mu(X_1,\dots,X_h)^{q-1},
\end{equation}
in which the product runs over nonzero vectors $a=(a_1,\dots,a_h)$ in $k^h$.  Second,
there is the identity
\begin{equation}
\label{piLT}
 [\pi]_{\LT}(\mu(X_1,\dots,X_n)) =  \det\left([\pi]_{\Fc^{\univ}}(X_i)\biggm\vert X_i^q \biggm\vert \cdots \biggm\vert X_i^{q^{h-1}}\right)_{1\leq i\leq h},
\end{equation}
valid in $\mathcal{A}[X_1,\dots,X_n]$.  This is easily seen by expanding the first column of the matrix according to Eq.~\eqref{piuniv}.
\begin{lemma}
\label{level1det}  If $X_1,\dots,X_h$ are sections of $\Fc^{\univ}[\pi]$, then $\mu(X_1,\dots,X_h)$ is a section of $\LT[\pi]$.  If the $X_i$ form a Drinfeld basis for $\Fc^{\univ}[\pi]$, then $\mu(X_1,\dots,X_h)$ constitutes a Drinfeld basis for $\LT[\pi]$.
\end{lemma}
\begin{proof}  Suppose $X_1,\dots,X_h$ are sections of $\Fc^{\univ}[\pi]$ over an $\mathcal{A}$-algebra $R$.   Then the claim that $\mu(X_1,\dots,X_h)$ is annihilated by $[\pi]_{\LT}$ follows from Eq.~\eqref{piLT}.  Now assume that $X_1,\dots,X_h$ is a Drinfeld basis for $\Fc^{\univ}[\pi]$.  This means that
\[\prod_{a\in k^h} \left(T-(a_1X_1+\dots+a_hX_h)\right)\text{ divides }[\pi]_{\Fc^{\univ}}(T)\]
in $R\ps{T}$, hence in $R[T]$.  Since $[\pi]_{\Fc^{\univ}}(T)$ is monic, these polynomials are equal:
\begin{equation}
\label{prodpi}
\prod_{a\in k^h} \left(T-(a_1X_1+\dots+a_hX_h)\right)=
\pi T+u_1T^q+\dots+
u_{h-1}T^{q^{h-1}}+T^{q^h}
\end{equation}
 Equating  coefficients of $T$ and using Eq.~\eqref{prodmu} shows that
\[ \mu(X_1,\dots,X_h)^{q-1} = (-1)^h \pi. \]
On the other hand,
\[\prod_{a\in k} (T-a\mu(X_1,\dots,X_h))=T^q-\mu(X_1,\dots,X_h)^{q-1}T=(-1)^{h-1}[\pi]_{\LT}(T),\]
which shows that $\mu(X_1,\dots,X_h)$ forms a Drinfeld basis for $\LT[\pi]\otimes R$.
\end{proof}

\subsection{\texorpdfstring{Good reduction of an affinoid in $\mathfrak{X}(\pi)$}{Good reduction of an affinoid in X(pi)}}
In this interlude we find an affinoid in $\mathfrak{X}(\pi)$ whose reduction is the Deligne-Lusztig variety for $\GL_h(k)$.  This is nothing new in light of~\cite{yoshida}, Prop. 6.15, but it will give a flavor of the corresponding calculation for $\mathfrak{X}(\pi^2)$.

\begin{prop} \label{tame} There is an isomorphism of local $\hat{\OO}_{F^{\nr}}$-algebras
\[
\frac{\hat{\OO}_{F^{\nr}}\ps{X_1,\dots,X_h}}{\mu(X_1,\dots,X_h)^{q-1}-(-1)^h\pi}
\tilde{\longrightarrow}
\mathcal{A}(\pi)
\]
carrying $X_i$ onto $X_i^{(1)}$.
\end{prop}
\begin{proof}
Let $\mathcal{A}(\pi)'=\hat{\OO}_{F^{\nr}}\ps{X_1,\dots,X_h}/(\mu(X_1,\dots,X_h)^{q-1}-(-1)^h\pi)$.  By Lemma~\ref{level1det} there is unique homomorphism $\mathcal{A}(\pi)'\to\mathcal{A}(\pi)$ of $\hat{\OO}_{F^{\nr}}$-algebras carrying $X_i$ onto $X_i^{(1)}$.   Since the $X_i^{(1)}$ form a system of regular local parameters of $\mathcal{A}(\pi)$, this homomorphism is surjective.   The algebra $\mathcal{A}(\pi)$ is a Galois extension of $\mathcal{A}$ with group $\GL_h(k)$.  But we can also furnish $\mathcal{A}(\pi)'$ with the structure of an $\mathcal{A}$-algebra, by identifying $u_i\in\mathcal{A}$ with the coefficient of $T^{q^i}$ on the left-hand side of Eq.~\eqref{prodpi}.  Then $\mathcal{A}(\pi)'$ becomes a Galois extension of $\mathcal{A}$ with group $\GL_h(k)$ as well, and the homomorphism $\mathcal{A}(\pi)'\to\mathcal{A}(\pi)$ respects the $\mathcal{A}$-algebra structure.  We conclude that $\mathcal{A}(\pi)'\to\mathcal{A}(\pi)$ is an isomorphism.
\end{proof}

Now let $E/F$ be the unramified extension of degree $h$, and let $E_1/E^{\nr}$ be the extension obtained by adjoining a root $\varpi$ of $X^{q^h-1}-(-1)^h\pi$.   Then $E_1/E^{\nr}$ is totally tamely ramified of degree $q^h-1$.
Let $\mathfrak{X}(1)^{\ts}\subset\mathfrak{X}(1)\otimes E_1$ be the affinoid polydisc defined by the conditions
\[ v(u_i)\geq v(\varpi^{q^h - q^i})=\frac{q^h-q^i}{q^h-1} \]
The notation is borrowed from~\cite{coleman:Xp3}:  This is exactly the domain on which $\Fc^{\univ}[\pi]$ admits no canonical subgroups;  {\em i.e.} where $\Fc^{\univ}$ is ``too supersingular".   Whenever $\Fc$ is a deformation of $\Sigma$ lying in $\mathfrak{X}(1)^{\ts}$, all nonzero roots of $\Fc[\pi]$ have valuation equal to $v(\varpi)$.  By applying the change of variables $X_i=\varpi V_i$ to Prop.~\ref{tame} we find:
\begin{Theorem}
The preimage of $\mathfrak{X}(1)^{\ts}$ in $\mathfrak{X}(\pi)\otimes E_1$ has reduction isomorphic to the smooth affine hypersurface over $\overline{k}$ with equation $\mu(V_1,\dots,V_h)^{q-1}=1$.
\end{Theorem}

\subsection{Determinants of structures of higher level.}
\label{dethigherlevel}

Now let $n\geq 1$, and suppose $X_1,\dots,X_h$ are sections of $\Fc^{\univ}[\pi^n]$.  We write $[\pi^a]_u(X)$ as an abbreviation for $[\pi^a]_{\Fc^{\univ}}(X)$.  We define the form $\mu_n$ by
\[ \mu_n(X_1,\dots,X_h)=\sum_{(a_1,\dots,a_h)} \mu\left([\pi^{a_1}]_u(X_1),\dots,[\pi^{a_h}]_u(X_h)\right), \]
where the sum runs over tuples of integers $(a_1,\dots,a_h)$ with $0\leq a_i\leq n-1$ whose sum is  $(h-1)(n-1)$.  It is clear that $\mu_n$ is $k$-multilinear and alternating in $X_1,\dots,X_h$.   Before proving that $\mu_n$ is $\OO_F$-linear, we will show:

\begin{prop} \label{mu_n} For sections $X_1,\dots,X_h$ of $\Fc^{\univ}[\pi^n]$, we have \[ [\pi]_{\LT}(\mu_n(X_1,\dots,X_h)) = \mu_{n-1}([\pi]_u(X_1),\dots,[\pi]_u(X_h)). \]  In particular $\mu_n(X_1,\dots,X_h)$ is a section of the group scheme $\LT[\pi^n]$.
\end{prop}
\begin{proof}
Let $a=(a_1,\dots,a_h)$ be a tuple of nonnegative integers.  Write $[\pi^a](X)$ for the tuple $([\pi^{a_1}]_{u}(X_1),\dots,[\pi^{a_h}]_{u}(X_h))$.   Applying Eq.~\eqref{piLT} we find
\begin{align*}
[\pi]_{\LT}(\mu \left([\pi^a](X)\right)) &= \det\left([\pi^{a_i+1}]_u(X_i) \biggm\vert [\pi^{a_i}]_u(X_i)^q \biggm\vert \cdots\biggm\vert [\pi^{a_i}]_u(X_i)^{q^{h-1}} \right) \\
&=\sum_{\sigma\in S_h} \sgn(\sigma) [\pi^{a_{\sigma(1)}+1}]_u\left(X_{\sigma(1)}\right) \prod_{j=1}^{h-1} [\pi^{a_{\sigma(j+1)}}]_u\left(X_{\sigma(j+1)}\right)^{q^j}
\end{align*}
Now assume the $X_i$ are sections of $\Fc^{\text{univ}}[\pi^n]$:  this means that the terms in the sum with $a_{\sigma(1)}=n-1$ vanish.   The expression $[\pi]_{\LT}(\mu_n(X_1,\dots,X_n))$ is thus a sum over pairs $(a,\sigma)$, where $\sigma\in S_h$ is a permutation and $a=(a_1,\dots,a_h)$ is a tuple of integers satisfying the conditions
\begin{enumerate}
\item $0\leq a_i\leq n-1$
\item $a_{\sigma(1)}<n-1$
\item $\sum_i a_i=(n-1)(h-1)$
\end{enumerate}
Let $b=(b_1,\dots,b_h)$ be the tuple defined by
\[ b_j = \begin{cases} a_j,& j=\sigma(1)\\ a_j - 1,& j\neq \sigma(1)   \end{cases}  \]
Note that each $b_i$ is nonnegative:  If $a_j=0$ for some $j\neq\sigma(1)$, the condition $\sum_ia_i=(n-1)(h-1)$ forces $a_k=n-1$ for all $k\neq j$, which implies that $a_{\sigma(1)}=n-1$, contradicting condition (ii) above.   As $(a,\sigma)$ runs over all pairs of tuples and permutations satisfying (1)--(3), the pair $(b,\sigma)$ runs over all pairs of tuples and permutations satisfying $0\leq b_i\leq n-2$ and $\sum_i b_i=(n-1)(h-1)-(h-1)=(n-2)(h-1)$.  We find

\begin{align*}
[\pi]_{\LT}\left(\mu_n(X_1,\dots,X_h)\right)&=
\sum_{(b,\sigma)} \sgn(\sigma)\prod_{j=1}^{h-1} [\pi^{b_{\sigma(j)+1}}]_u\left(X_{\sigma(j)}\right)^{q^j} \\
&=\sum_b \mu([\pi^{b_1+1}]_u(X_1),\dots,[\pi^{b_h+1}]_u(X_h)) \\
&=\mu_{n-1}\left([\pi]_u(X_1),\dots,[\pi]_u(X_h)\right)
\end{align*}
as required.
\end{proof}

Now we can establish the $\OO_F$-linearity of $\mu_n$.  For this it suffices to show that
\[\mu_n([\pi]_u(X_1),X_2,\dots,X_{h-1})=[\pi]_{\LT}(\mu_n(X_1,\dots,X_h)).\]
We have
\[
\mu_n([\pi]_u(X_1),X_2,\dots,X_{h-1}) = \sum_a \mu([\pi^a](X)),
\]
where $a=(a_1,\dots,a_{h-1})$ runs over tuples satisfying $1\leq a_1\leq n-1$, $0\leq a_i\leq n-1$ for $i>1$, and $\sum_i a_i=(h-1)(n-1)+1$.  But these conditions force $a_i\geq 1$ for $i=1,\dots,h$.  Write $a_i=b_i+1$, so that $0\leq b_i\leq n-2$ and $\sum_i b_i=(h-1)(n-1)$.  Then
\begin{align*}
\mu_n([\pi]_u(X_1),X_2,\dots,X_{h-1})&=
\sum_b \mu([\pi^{b_1+1}]_u(X_1),\dots,[\pi^{b_h+1}]_u(X_h)) \\
&= \mu_{n-1}([\pi]_u(X_1),\dots,[\pi]_u(X_h)) \\
&=[\pi]_{\LT}(\mu_n(X_1,\dots,X_h))
\end{align*}
by Prop.~\ref{mu_n}.

We have established part (1) of Thm.~\ref{functorial}.  Part (1) allows us to reduce part (2) to the case of $n=1$, which has already been treated in Prop.~\ref{level1det}.

Recall that $X_1^{(n)},\dots,X_h^{(n)}$ are the canonical coordinates on $\mathfrak{X}(\pi^n)$.   Thm.~\ref{functorial} shows that the function $\Delta^{(n)}=\mu_n(X_1^{(n)},\dots,X_h^{(n)})$ is a nonzero root of $[\pi^n]_{\LT}(T)$.   The following simple lemma will be useful in the next section.

\begin{lemma}
\label{trM}
Let $M\in M_h(\OO_F/\pi^n\OO_F)$ be a matrix.  Then
\[ \mu_n(M(X_1^{(n)}),\dots,X_h^{(n)})+\dots+\mu_n(X_1^{(n)},\dots,M(X_h^{(n)}))
= [\tr M]_{\LT}(\Delta^{(n)}). \]
\end{lemma}

\section{An affinoid with good reduction}

We now reach the technical heart of the paper.   In this section we will construct an  open affinoid neighborhood $\mathfrak{Z}$ around an unramified canonical point $x$  whose reduction is as in Thm.~\ref{mainthm}.  These affinoids appear as connected components of the preimage of a subdisc $\mathfrak{X}(1)^1$ inside of the polydisc $\mathfrak{X}(1)$.  The polydisc $\mathfrak{X}(1)^1$ is small enough so that the local system $\Fc^{\text{univ}}[\pi]$ may be trivialized over $\mathfrak{X}(1)^1$, which is to say that the quotient map $\mathfrak{X}(\pi)\to\mathfrak{X}(1)$ admits a section over $\mathfrak{X}(1)^1$.  An approximation to this section is computed explicitly in \S\ref{analsects}.   A consequence is that the preimage of $\mathfrak{X}(1)^1$ in $\mathfrak{X}(\pi)$ is a disjoint union of polydiscs $\mathfrak{X}(\pi)^{1,x}$ indexed by the canonical points of $\mathfrak{X}(\pi)$.

In \S\ref{invcoord} we turn to the space $\mathfrak{X}(\pi^2)$.  An unramified canonical point $x\in\mathfrak{X}(\pi^2)$ determines a subgroup $K_{x,2}$ of $\GL_h(\OO_F)$ lying properly between $1+\pi M_h(\OO_F)$ and $1+\pi^2 M_h(\OO_F)$. Let \[\mathfrak{X}(K_{x,2})=\mathfrak{X}(\pi^2)/K_{x,2}.\]  Then the affinoid $\mathfrak{Z}$ of Thm.~\ref{mainthm} is the preimage of $\mathfrak{X}(\pi)^{1,x}$ in $\mathfrak{X}(K_{x,2})$.  We introduce a family of coordinates $Y(\zeta)$ on $\mathfrak{X}(\pi^2)$ which are invariant under $K_{x,2}$, one for each $\zeta$ in $\OO_E$.    (The formation of the $Y(\zeta)$ is modeled on the determinant functor $\mu_2$ from \S\ref{determinants}.)  Thus the $Y(\zeta)$ are analytic functions on $\mathfrak{X}(K_{x,2})$;  it turns out (Prop.~\ref{Yzeta}) that the $Y(\zeta)$ are integral functions on $\mathfrak{Z}$.  A simple linear combination $Y$ of the coordinates $Y(\zeta)$ generates the ring of integral analytic functions on $\mathfrak{Z}$ as an algebra over the ring of integral analytic functions on the polydisc $\mathfrak{X}(\pi)^{x,1}$.   The
equation for the reduction $\overline{\mathfrak{Z}}$ follows from the congruence calculated in Prop.~\ref{Yprop}.

We often work with affinoid algebras $\mathcal{B}$ over a field $E$, where $E/F$ is a finitely ramified extension contained in the completion of the separable closure of $F$.  For $f\in\mathcal{B}$ we write $v(f)$ for the infimum of $v(f(z))$ as $z$ runs though $\Spm\mathcal{B}$.

\subsection{\texorpdfstring{Analytic sections of $\Fc^{\univ}[\pi]$}{Analytic sections of F[pi]}}
\label{analsects}
Let $E/F$ be the unramified extension of degree $h$, so that $\OO_E=k_h\ps{\pi}$.  Let $\Fc_0$ be the deformation obtained by specializing the variables $u_i$ to 0 in $\Fc^{\univ}$, so that $[\pi]_{\Fc_0}(X)=\pi X+X^{q^h}$.   Then $\Fc_0$ admits endomorphisms by $\OO_E$.  As a formal $\OO_E$-module, $\Fc_0$ has height 1.  We will denote by $x^{(0)}$ the unramified canonical point in $\mathfrak{X}(1)$ corresponding to $\Fc_0$.

For $n\geq 1$, let $E_n$ be the extension of $\hat{E}^{\nr}$ given by adjoining the roots of $[\pi^n]_{\Fc_0}(X)$.  Thus the preimages of $x^{(0)}$ in $\mathfrak{X}(\pi)$ are the points $x=x^{(1)}\in \mathfrak{X}(\pi)$ corresponding to Drinfeld bases $x_1,\dots,x_h\in \gp_{E_1}$ for $\Fc_0[\pi]$.
Let $\mathfrak{X}(1)^{1}\subset\mathfrak{X}(1)$ be the affinoid neighborhood defined by the conditions $v(u_i)\geq 1$, $i=1,\dots,h-1$.  Let $V_i=\pi^{-1}u_i$, so that the $V_i$ are a chart of integral coordinates on $\mathfrak{X}(1)^1$.  The ring of integral analytic functions on $\mathfrak{X}(1)^1$ is therefore  $\hat{\OO}_{F^{\nr}}\ta{V_1,\dots,V_{h-1}}$.

We claim that over $\mathfrak{X}(1)^1\otimes E_1$, the local system $\Fc^{\univ}[\pi]$ may be trivialized.  This means that every nonzero torsion point of $\Fc_0[\pi]$ can be ``spread out" to a unique section of $\Fc^{\univ}[\pi]$ over $\mathfrak{X}(1)^1\otimes E_1$.
To be precise:
\begin{prop}
\label{trivialization}  The preimage of $\mathfrak{X}(1)^1\otimes E_1$ in $\mathfrak{X}(\pi)\otimes E_1$ is the disjoint union of polydiscs $\mathfrak{X}(\pi)^{1,x}$ over $E_1$, each containing a unique unramified canonical point $x$.  For such a point $x$, corresponding to the basis $x_1,\dots,x_h$ of $\Fc_0[\pi]$, we have the following congruence, valid in the ring of integral analytic functions on $\mathfrak{X}(\pi)^{1,x}$:
\begin{equation}
\label{xr1}
X_r^{(1)}\equiv (-1)^{h-1}\det
\begin{pmatrix}
  V_1 & V_2 &\cdots & V_{h-1} & x_r\\
  1  & V_1^q &\cdots & V_{h-2}^q & x_r^q +\pi x_rV_{h-1}^q\\
   0 & 1 & \cdots & V_{h-3}^{q^2} & x_r^{q^2} +\pi x_rV_{h-2}^{q^2}\\
 \vdots & & \ddots & &\vdots \\
  0 & 0 &\cdots & 1& x_r^{q^{h-1}}+\pi x_rV_1^{q^{h-1}}
\end{pmatrix}
\end{equation}
modulo $\pi^{q-1+\frac{q}{q^h-1}}$.
\end{prop}
\begin{proof}
Let $x_1,\dots,x_h$ be a basis of $\Fc_0[\pi]$.  Consider the polynomial $[\pi]_{\Fc^{\univ}}(X)= \pi X +\pi V_1 X^q+\dots+ \pi V_{h-1}X^{q^{h-1}} + X^{q^h}\in \OO_F\ta{V_1,\dots,V_{h-1}}[X]$.    By studying the Newton polygon of the translate $[\pi]_{\Fc^{\univ}}(X-x_r)$, we find that there is a unique root $X_r\in \OO_{E_1}\ta{V_1,\dots,V_{h-1}}$ of $[\pi]_{\Fc^{\univ}}(X)$ for which $v(X_r-x_r)>v(x_r)=1/(q^h-1)$.  This root satisfies $v(X_r-x_r)=v(x_r^q)=q/(q^h-1)$.    Then $v(X_r-x_s)=1/(q^h-1)$ for $r\neq s$.  This already implies that the preimage of $\mathfrak{X}(1)^1\otimes E_1$ in $\mathfrak{X}(\pi)\otimes E_1$ is the union of polydiscs $\mathfrak{X}(\pi)^{1,x}$, where $\mathfrak{X}(\pi)^{1,x}$ is the affinoid described by the inequalities $v(X_r^{(1)}-x_r)\geq v(x_r^q)$, $r=1,\dots,h$.

Now let $D\in\OO_{E_1}[V_1,\dots,V_{h-1}]$ be the expression on the right hand side of Eq.~\eqref{xr1}.  Expand the determinant in Eq.~\eqref{xr1} along its first row and label the minors $A_1,\dots A_h$, signed appropriately so that
\begin{equation}
\label{minors}
D=\sum_{i=1}^{h-1} V_iA_i+ x_rA_h.
\end{equation}
That is,
\begin{equation}
\label{minorsformula}
A_i = (-1)^{h-i} \det
\begin{pmatrix}
V_1^{q^i} & V_2^{q^i} &\cdots & V_{h-i-1}^{q^i} & x_r^{q^i} + \pi x_rV_{h-i}^{q^i}\\
1 & V_1^{q^i} & \cdots   & V_{h-i-2}^{q^{i+1}} & x_r^{q^{i+1}} + \pi x_rV_{h-i-1}^{q^{i+1}} \\
0 & 1 & \cdots & V_{h-i-3}^{q^{i+2}} & x_r^{q^{i+2}} +\pi x_rV_{h-i-2}^{q^{i+2}} \\
\vdots &  & \ddots & \vdots & \vdots \\
0 & 0 & \cdots & 1 & x_r^{q^{h-1}}+\pi x_rV_1^{q^{h-1}}
\end{pmatrix}
\end{equation}
for $i=1,\dots,h-1$, and $A_h=1$.

In order to complete the proof of Prop.~\ref{trivialization}, we will show that $[\pi]_{\Fc^{\univ}}(D)$ is sufficiently close to 0 to ensure the congruence in Eq.~\eqref{xr1}.   

Observe that for $i=1,\dots,h-1$ we have the following congruence modulo $\pi^{q+\frac{q}{q^h-1}}$:

{\small
\[D^{q^i} \equiv (-1)^{h-1}\det
\begin{pmatrix}
V_1^{q^i} & V_2^{q^i} & \cdots & V_{h-i}^{q^i} & V_{h-i+1}^{q^i} & \cdots &V_{h-1}^{q^i} & x_r^{q^i}\\
1 & V_1^{q^{i+1}} & \dots & V_{h-i-1}^{q^{i+1}}& V_{h-i}^{q^{i+1}} & \cdots &V_{h-2}^{q^{i+1}} & x_r^{q^{i+1}} \\
\vdots & & \ddots & \vdots & \vdots & & & \vdots \\
0 & 0 & \dots & V_1^{q^{h-1}} & V_2^{q^{h-1}} & \dots & V_{i}^{q^{h-1}} & x_r^{q^{h-1}}\\
0 & 0 & \dots & 1 & V_1^{q^h} & \dots & V_{i-1}^{q^{h}} & -\pi x_r\\
0 & 0 & \dots & 0 & 1 & \dots & V_{i-2}^{q^{h+1}} & 0 \\
\vdots & & & \vdots & \vdots & \ddots & \vdots & \vdots \\
0 & 0 & \dots & 0 & 0 & \dots & 1 & 0
\end{pmatrix}\]
}

Placing the final column of this matrix into position $(h-i+1)$ transforms the above matrix into one of the form $\tbt{A}{B}{0}{C}$, where $A$ is a matrix with dimensions $(h-i+1)\times (h-i+1)$ and $C$ is an upper triangular matrix with 1s along the diagonal.  We find
{\small
\begin{equation}
\label{xrqi}
 D^{q^i}\equiv (-1)^{h+i}\det
\begin{pmatrix}
V_1^{q^i} & V_2^{q^i} & \cdots    & V_{h-i-1}^{q^i} & V_{h-i}^{q^i} & x_r^{q^i} \\
1              & V_1^{q^{i+1}} & \cdots & V_{h-i-2}^{q^{i+1}} & V_{h-i-1}^{q^{i+1}} & x_r^{q^{i+1}} \\
0              & 1                     & \cdots & V_{h-i-3}^{q^{i+2}} & V_{h-i-2}^{q^{i+2}} & x_r^{q^{i+2}} \\
\vdots      &                         & \ddots&                              & \vdots\\
0             & 0                       & \cdots& 1& V_1^{q^{h-1}}      & x_r^{q^{h-1}} \\
0             &  0                      & \cdots& 0 & 1                          & -\pi x_r
\end{pmatrix}\kern-5pt\pmod{\pi^{q+\frac{q}{q^h-1}}}.
\end{equation}
}
We can apply elementary row operations to use the 1 in column $h-i$ of this matrix to cancel the entries above it.  When this is done, we find
\begin{equation}
\label{DeqA}
D^{q^i}\equiv -A_i\pmod{\pi^{q+\frac{q}{q^h-1}}},\;i=1,\dots,h-1
\end{equation}
where $A_1,\dots,A_{h-1}$ are the minors from Eq.~\eqref{minorsformula}.
We also have
\begin{equation}
\label{xrqh}
D^{q^h}\equiv -\pi x_r\equiv -\pi x_rA_h \pmod{\pi^{q+\frac{q}{q^h-1}}}.
\end{equation}
Combining Eqs.~\eqref{minors},~\eqref{DeqA} and~\eqref{xrqh} gives
\begin{eqnarray*}
[\pi]_{\Fc^{\univ}}(D)
&=&\pi D+\pi V_1D^q + \dots +\pi V_{h-1}D^{q^{h-1}}+D^{q^h}\\
&\equiv& \pi D - \pi (V_1A_1+\dots+V_{h-1}A_{h-1}+x_rA_h)\\
&\equiv&0\pmod{\pi^{q+\frac{q}{q^h-1}}}.
\end{eqnarray*}

The ring of integral analytic functions on the polydisc $\mathfrak{X}(\pi)^{1,x}$ is $\OO_{E_1}\ta{V_1,\dots,V_h}$.  In this ring we have the congruences $D\equiv X_r^{(1)} \equiv x_r\pmod{x_r^q}$.   Let $Y=D-X_r^{(1)}$.  Then $Y\equiv 0\pmod{x_r^q}$ and $[\pi]_{\Fc^{\univ}}(Y)\equiv 0\pmod{\pi^{q+1/(q^h-1)}}$.  Examining the Newton polygon of $[\pi]_{\Fc^{\univ}}(X)$ shows that $Y\equiv 0\pmod{\pi^{q-1+1/(q^h-1)}}$.
\end{proof}

\subsection{\texorpdfstring{Some invariant coordinates on $\mathfrak{X}(\pi^2)$.}{Some invariant coordiates on X(pi2)}}
\label{invcoord}

Choose a compatible system of bases $x_1^{(n)},\dots,x_h^{(n)}$ for $\Fc_0[\pi^n]$, $n\geq 1$.  This is tantamount to choosing a compatible system of unramified canonical points $x^{(n)}\in\mathfrak{X}(\pi^n)$ lying above the point $x^{(0)}\in\mathfrak{X}(1)$ corresponding to the deformation $\Fc_0$.
Since $\Fc_0$ admits $\OO_F$-linear endomorphisms by $\OO_E$, our choice of compatible system induces an embedding of $\OO_E$ into $\mathfrak{A}=M_h(\OO_F)$, and we identify $\OO_E$ with its image.  For $M\in \mathfrak{A}$, recall the definition of $M(X_i^{(n)})$ from Eq.~\eqref{MX}.  We have $\zeta(X_i^{(n)})(x^{(n)}) = \zeta x_i^{(n)}$ for $i=1,\dots,h$, $\zeta\in k_h$.

The unit group $\mathfrak{A}^\times=\GL_h(\OO_F)$ has the usual filtration $U^n_{\mathfrak{A}}=1+\gp^n\mathfrak{A}$, $n\geq 1$.   Let $C\subset \mathfrak{A}$ be the orthogonal complement of $\OO_E$ under the standard trace pairing, and let $\gp_E$ be the maximal ideal of $\OO_E$.  Define a subgroup $K_{x,2}$ of $\mathfrak{A}^\times$ by
\[ K_{x,2} = 1+\gp_E^2  + \gp_E C,\]
so that $K_{x,2}$ lies between $U_{\mathfrak{A}}^1$ and $U_{\mathfrak{A}}^2$.  In what follows we will assume the choice of $x$ is fixed and write simply $K_2$.   Write $\mathfrak{X}(K_2)$ for the quotient of $\mathfrak{X}(\pi^2)$ by $K_2$.

We shall construct an alternating $k$-linear expression $Y$ in the canonical coordinates $X_1^{(2)},\dots,X_h^{(2)}$ which is fixed by $K_2$, so that it descends to an analytic function on $\mathfrak{X}(K_2)$.  It happens that $Y$ satisfies a polynomial equation with coefficients in $\OO_{E_2}\ta{V_1,\dots,V_h}$ whose reduction modulo the maximal ideal of $\OO_{E_2}$ gives the smooth hypersurface of Thm.~\ref{mainthm}.

We continue using the shorthand $X_r=X_r^{(1)}$.   We introduce the new shorthand $Y_r=X_r^{(2)}$,  so that $[\pi]_{\Fc^{\univ}}(Y_r)=X_r$.  Also we let $\Delta=\Delta^{(1)}=\mu(X_1,\dots,X_h)$;  this is a locally constant function satisfying $\Delta^{q-1}=(-1)^h\pi$.  For $\zeta\in \OO_E$, let
\[ W(\zeta)=\mu(\zeta(Y_1),X_2,\dots,X_h)
+\dots+\mu(X_1,X_2,\dots,\zeta(Y_h)). \]
Note that $W(1)=\mu_2(X_1,\dots,X_h)=\Delta^{(2)}$.   We record the action of $U_{\mathfrak{A}}^1$ on the functions $W(\zeta)$:  For $g=1+\pi M\in U_{\mathfrak{A}}^1$, we have
\begin{equation}
\label{gaction}
g(W(\zeta)) = W(\zeta) + [\tr (M\zeta)]_{\LT}(\Delta)
\end{equation}
by Lemma~\ref{trM}.   It follows that $W(\zeta)$ is invariant under $K_2$, and that $[\pi]_{\LT}(W(\zeta))$ is invariant under $U_{\mathfrak{A}}^1$, so that $[\pi]_{\LT}(W(\zeta))$ belongs to $\mathcal{A}(\pi)$.  We can see this directly:  by Eq.~\eqref{piLT} we have
\begin{equation}
\label{Wzetamatrix}
[\pi]_{\LT}(W(\zeta))=\det
\begin{pmatrix}
\zeta(X_1) & X_1^q & \cdots & X_1^{q^{h-1}} \\
\vdots        & \vdots & \ddots & \vdots \\
\zeta(X_h) & X_h^q & \cdots & X_h^{q^{h-1}}
\end{pmatrix},
\end{equation}
which visibly belongs to $\mathcal{A}(\pi)$.

We will use the symbol $x$ to denote our compatible system of canonical points $x^{(n)}\in \mathfrak{X}(\pi^n)$.  Then $f(x)$ is well-defined when $f$ is an analytic function on any of the spaces $\mathfrak{X}(\pi^n)$.  We will use $\mathfrak{X}(\pi)^{1,x}$ to refer to the polydisc constructed in \S\ref{analsects} using the canonical point $x^{(1)}$.

By Prop.~\ref{trivialization}, the restriction of the function $[\pi]_{\LT}(W(\zeta))$ to $\mathfrak{X}(\pi)^{1,x}$ lies in $\OO_{E_1}\ta{V_1,\dots,V_h}$, where we recall that the variables $V_r=\pi^{-1}u_r$ form our chart of integral coordinates on $\mathfrak{X}(1)^{1}$.   Let $\mathfrak{Z}$ be the preimage of the polydisc $\mathfrak{X}(\pi)^{1,x}$ in $\mathfrak{X}(K_2)\otimes E_2$.  It will be useful to transform the functions $W(\zeta)$ into integral functions $Y(\zeta)$ on $\mathfrak{Z}$ for which $\abs{Y(\zeta)}_{\mathfrak{Z}}=1$.   Let $w(\zeta)=W(\zeta)(x)$, and let
\begin{equation}
\label{Yzetadef}
Y(\zeta) = (-1)^{h-1}\frac{W(\zeta) - w(\zeta)}{\Delta}.
\end{equation}

\begin{prop}
\label{Yzeta}
There exists $\eps>0$ for which the congruence
\[
Y(\zeta)^q - Y(\zeta) \equiv
\begin{pmatrix}
V_1 & V_2 & \dots & V_{h-1} & 0 \\
1     &  V_1^q & \dots & V_{h-2}^q & (\zeta^q-\zeta)V_{h-1}^q \\
0    &   1   &  \dots  &  V_{h-3}^{q^2}  & (\zeta^{q^2}-\zeta)V_{h-2}^{q^2} \\
\vdots & & \ddots& & \vdots  \\
0 &  0 & \cdots & 1 & (\zeta^{q^{h-1}}-\zeta)V_1^{q^{h-1}}
\end{pmatrix}\pmod{\pi^\eps}\]
is valid in the ring of integral analytic functions on $\mathfrak{Z}$.
\end{prop}
\begin{proof}
The idea is to apply Prop.~\ref{trivialization} to Eq.~\eqref{Wzetamatrix}.  In preparation for this, we need some determinant identities.  For $i=1,\dots,h$, let $B_i\in k[V_1,\dots,V_{h-1}]$ be $(-1)^i$ times the determinant of the top left $i\times i$ submatrix of
\[
\begin{pmatrix}
V_1 & V_2 & \cdots & V_{h-1} & 0 \\
1   & V_1^q & \cdots & V_{h-2}^q & V_{h-1}^q \\
0 &  1 & \cdots & V_{h-3}^{q^2} & V_{h-2}^{q^2} \\
\vdots & & \ddots & & \vdots \\
0 & 0 & \cdots & 1 & V_1^{q^{h-1}}
\end{pmatrix}
\]
Curiously, the transformation $(V_1,\dots,V_{h-1})\mapsto (B_1,\dots,B_{h-1})$ is an involution.   That is, the determinant of the top left $i\times i$ submatrix of
\[
\begin{pmatrix}
B_1 & B_2 & \cdots & B_{h-1} & 0 \\
1   & B_1^q & \cdots & B_{h-2}^q & B_{h-1}^q \\
0 &  1 & \cdots & B_{h-3}^{q^2} & B_{h-2}^{q^2} \\
\vdots & & \ddots & & \vdots \\
0 & 0 & \cdots & 1 & B_1^{q^{h-1}}
\end{pmatrix}
\]
is $(-1)^iV_i$:  this can be proven by induction on $i$.  This implies the following identity, valid in the polynomial ring $k[V_1,\dots,V_{h-1},z_1,\dots,z_{h-1}]$:

\begin{eqnarray}
\nonumber
\label{BV}\\
&&\nonumber
\kern-40pt\det
\begin{pmatrix}
z_1B_1 & z_2B_2 & \cdots & z_{h-1}B_{h-1} & 0 \\
1  &  B_1^q & \cdots & B_{h-2}^q & B_{h-1}^q \\
0  &  1 &  \cdots & B_{h-3}^{q^2} & B_{h-2}^{q^2} \\
\vdots & & \ddots & & \vdots \\
0 & 0 & \ddots & 1 & B_1^{q^{h-1}}
\end{pmatrix}\\
&=&
\det
\begin{pmatrix}
V_1 & V_2 & \cdots & V_{h-1} & 0 \\
1   & V_1^q & \cdots & V_{h-2}^q & z_1V_{h-1}^q \\
0 &  1 & \cdots & V_{h-3}^{q^2} & z_2V_{h-2}^{q^2} \\
\vdots & & \ddots & & \vdots \\
0 & 0 & \cdots & 1 & z_{h-1}V_1^{q^{h-1}}
\end{pmatrix}
\end{eqnarray}
This is because both expressions equal
\[ z_1B_1V_{h-1}^q + z_2B_2V_{h-2}^{q^2}+\dots+z_{h-1}B_{h-1}V_1^{q^{h-1}}.\]

According to Prop.~\ref{trivialization}, the coordinate $X_r$ may be expressed modulo $\pi^{q-1+\frac{q}{q^h-1}}$ as a linear combination of the powers $x_r,\dots,x_r^{q^{h-1}}$:
\begin{equation}
X_r\equiv (1-\pi B_h)x_r+B_1x_r^q+B_2x_r^{q^2}+\dots+B_{h-1}x_r^{q^{h-1}}\pmod{\pi^{q-1+\frac{q}{q^h-1}}}.
\end{equation}
For $\zeta\in k_h$ we have
{\small
\begin{equation}
\label{zetaxr}
\zeta(X_r)\equiv \zeta (1-\pi B_h)x_r + \zeta^q B_1 x_r^q+\zeta^{q^2}B_2x_r^{q^2} + \dots+\zeta^{q^{h-1}} B_{h-1}x_r^{q^{h-1}}\pmod{\pi^{q-1+\frac{q}{q^h-1}}}
\end{equation}
}
Also, for $i=1,\dots,h-1$ we have
\begin{equation}
\label{XrBi}
X_r^{q^i} \equiv -\pi B_{h-i}^{q^i} x_r +x_r^{q^i} + B_1^{q^i}x_r^{q^{i+1}} + \dots + B_{h-1-i}^{q^i} x_r^{q^{h-1}}\pmod{\pi^{N}},
\end{equation}
where $N\geq q+\frac{q}{q^h-1}$.
Eqs.~\eqref{zetaxr} and~\eqref{XrBi} may be combined into the congruence of matrices
\begin{eqnarray}
\nonumber
\label{XBzeta}
\\
\nonumber
&&
\kern-40pt\begin{pmatrix}
\zeta(X_1)+E_1 & \cdots & \zeta(X_h)+E_h  \\
X_1^q & \cdots         & X_h^q \\
\vdots & \ddots & \vdots \\
X_1^{q^{h-1}} & \cdots  & X_h^{q^{h-1}}
\end{pmatrix}\\
\nonumber
&\equiv&
\begin{pmatrix}
\zeta (1-\pi B_h) & \zeta^q B_1 & \zeta^{q^2} B_2 & \dots & \zeta^{q^{h-1}} B_{h-1} \\
-\pi B_{h-1}^q & 1 & B_1^q & \cdots & B_{h-2}^q \\
-\pi B_{h-2}^{q^2} & 0 & 1 & \cdots & B_{h-3}^{q^2} \\
\vdots & & & \ddots & \vdots \\
-\pi B_1^{q^{h-1}} & 0 & 0 & \cdots & 1
\end{pmatrix}\\
&\times&
\begin{pmatrix}
x_1 &\cdots & x_h \\
\vdots & \ddots & \vdots\\
x_1^{q^{h-1}}  &\cdots & x_h^{q^{h-1}}
\end{pmatrix}
\end{eqnarray}
modulo $\pi^{N}$, where $v(E_i)\geq q-1+q/(q^h-1)$.   We take determinants of both sides of Eq.~\eqref{XBzeta}.  On the left hand side, we apply Eq.~\eqref{Wzetamatrix}:  the determinant is congruent to $[\pi]_{\LT}(W(\zeta))$ modulo an error term $\pi^\delta$, of valuation
\[ \delta\geq q-1+\frac{q}{q^h-1}+\frac{q+q^2+\dots+q^{h-1}}{q^h-1} = q-1+\frac{q-1}{q^h-1} + \frac{1}{q-1}. \]
On the right hand side, the determinant is $\Delta$ times
\[
\zeta - \zeta\pi B_h + (-1)^h\pi\det
\begin{pmatrix}
 \zeta^q B_1 & \zeta^{q^2} B_2 & \dots &  \zeta^{q^{h-1}} B_{h-1} & 0  \\
 1 & B_1^q & \cdots & B_{h-2}^q & B_{h-1}^q \\
 0 &1  &  \cdots & B_{h-3}^{q^2} & B_{h-2}^{q^2}\\
\vdots & & \ddots & & \vdots \\
0 & 0& \cdots & 1 & B_1^{q^{h-1}}
\end{pmatrix},
\]
and by the identity in Eq.~\ref{BV} this equals
\[
\zeta + (-1)^h\pi
\det
\begin{pmatrix}
V_1 & V_2 & \dots & V_{h-1} & 0 \\
1     &  V_1^q & \dots & V_{h-2}^q & (\zeta^q-\zeta)V_{h-1}^q \\
0    &   1   &  \dots  &  V_{h-3}^{q^2}  & (\zeta^{q^2}-\zeta)V_{h-2}^{q^2} \\
\vdots & & \ddots& & \vdots  \\
0 &  0 & \cdots & 1 & (\zeta^{q^{h-1}}-\zeta)V_1^{q^{h-1}}
\end{pmatrix}.\]
Equating determinants of both sides of Eq.~\eqref{XBzeta} now yields
$$\displaylines{
[\pi]_{\LT}(W(\zeta)) \equiv \hfill\cr
\hfill \equiv \zeta\Delta +(-1)^h\pi\Delta
\det
\begin{pmatrix}
V_1 & V_2 & \dots & V_{h-1} & 0 \\
1     &  V_1^q & \dots & V_{h-2}^q & (\zeta^q-\zeta)V_{h-1}^q \\
0    &   1   &  \dots  &  V_{h-3}^{q^2}  & (\zeta^{q^2}-\zeta)V_{h-2}^{q^2} \\
\vdots & & \ddots& & \vdots  \\
0 &  0 & \cdots & 1 & (\zeta^{q^{h-1}}-\zeta)V_1^{q^{h-1}}
\end{pmatrix}\pmod{\pi^\delta}
}$$

The functions $V_1,\dots,V_{h-1}$ vanish at the canonical point $x$;  therefore so do the functions $B_1,\dots,B_{h-1}$.  Applying the above congruence to $x$ gives
\begin{equation}
\label{w}
[\pi]_{\LT}(w(\zeta))\equiv \zeta\Delta\pmod{\pi^\delta}.
\end{equation}
We have $W(\zeta) = w(\zeta) +(-1)^{h-1}\Delta Y(\zeta)$, so that
\[ [\pi]_{\LT}(W(\zeta)) = [\pi]_{\LT}(w(\zeta)) + (-1)^h\pi\Delta(Y(\zeta)^q - Y(\zeta)) \]
Therefore the congruence claimed in the proposition is valid modulo $\pi^\eps$, where
\[\eps = \delta - 1 -\frac{1}{q-1}\geq q-2 +\frac{q-1}{q^h-1} >0.\]
\end{proof}
The functions $Y(\zeta)$ on $\mathfrak{Z}$ each generate a degree $q$ algebra over the field of meromorphic functions on the polydisc $\mathfrak{X}(\pi)^{1,x}$.  But  the morphism $\mathfrak{Z}\to\mathfrak{X}(\pi)^{1,x}\otimes E_2$ has degree $q^h$.  We will now construct a linear combination of the $Y(\zeta)$ which generates the entire ring of integral analytic functions on $\mathfrak{Z}$ as an algebra over $\OO_{E_2}\ta{V_1,\dots,V_{h-1}}$.

Let $\zeta,\zeta^q,\dots,\zeta^{q^h}$ be a basis for $k_h/k$, and let $\beta\in k_h$ be such that
\begin{equation}
\label{traceident}
 \tr_{k_h/k} (\beta\zeta^{q^i}) = \begin{cases} 1,& i=0,\\ 0& i=1,\dots,h-1. \end{cases}
\end{equation}
This implies that $\beta,\dots,\beta^{q^{h-1}}$ is a basis for $k_h/k$ as well.
Let
\begin{equation}
\label{Ydef}
Y=\sum_{i=0}^{h-1} \beta^{q^i} Y(\zeta^{q^i}).
\end{equation}
Then the stabilizer of $Y$ in $U_{\mathfrak{A}}^1$ is exactly $K_2$.
\begin{prop}
\label{Yprop}
There exists $\eps>0$ for which the congruence
\begin{equation}
\label{YV}
Y^{q^h}-Y \equiv
\begin{pmatrix}
V_1^{q^h}-V_1 & V_2^{q^h}-V_2  & \cdots & V_{h-1}^{q^h} - V_{h-1} & 0 \\
1 & V_1^q& \cdots & V_{h-2}^q & V_{h-1}^q \\
0 & 1 &\cdots & V_{h-3}^{q^2} & V_{h-2}^{q^2} \\

\vdots &  & \ddots & & \vdots \\

0 & 0 &  \cdots& 1 & V_1^{q^{h-1}}

\end{pmatrix}
\end{equation}
holds modulo $\pi^\eps$ in the ring of integral analytic functions on $\mathfrak{Z}$. \end{prop}
\begin{proof}
We have
{\small
\begin{align*}
Y^{q^h}-Y
 &=\sum_{j=0}^{h-1} \beta^{q^j} (Y(\zeta^{q^j})^{q^h}-Y(\zeta^{q^j})) \\
&=\sum_{j=0}^{h-1}\beta^{q^j}\sum_{i=0}^{h-1} (Y(\zeta^{q^j})^q-Y(\zeta^{q^j}))^{q^i}\\
&\equiv\sum_{i=0}^{h-1} \sum_{j=0}^{h-1} \beta^{q^j}\det
\begin{pmatrix}
V_1^{q^i} & V_2^{q^i} & \dots & V_{h-1}^{q^i} & 0 \\
1     &  V_1^{q^{i+1}} & \dots & V_{h-2}^{q^{i+1}} & (\zeta^{q^{i+j+1}}-\zeta^{q^{i+j}})V_{h-1}^{q^{i+1}} \\
0    &   1   &  \dots  &  V_{h-3}^{q^{i+2}}  & (\zeta^{q^{i+j+2}}-\zeta^{q^{i+j}})V_{h-2}^{q^{i+2}} \\
\vdots & & \ddots& & \vdots  \\
0 &  0 & \cdots & 1 & (\zeta^{q^{i+j+h-1}}-\zeta^{q^{i+j}})V_1^{q^{h-1}}
\end{pmatrix}
\end{align*}
}
modulo $\pi^\eps$, by Prop.~\ref{Yzeta}.  We now apply the orthogonality relations in Eq.~\eqref{traceident}.  The term with $i=0$ is $(-1)^{h-1}B_h$, and the term with $1\leq i\leq h-1$ is
\[
\det
\begin{pmatrix}
V_1^{q^i} & V_2^{q^i} & \cdots & V_{h-i}^{q^i} & \cdots & V_{h-1}^{q^i} & 0 \\
1     &  V_1^{q^{i+1}} & \cdots & V_{h-i-1}^{q^{i+1}} & \cdots & V_{h-2}^{q^{i+1}} & 0 \\
0    &   1   &  \cdots & V_{h-i-2}^{q^{i+2}} & \cdots &  V_{h-3}^{q^{i+2}}  & 0 \\
\vdots & & & & & & \vdots  \\
0  & 0 & \cdots & V_1^{q^{h-1}} & \cdots & V_i^{q^{h-1}} & 0 \\
0  & 0 & \cdots & 1 & \cdots & V_{i-1}^{q^h} & V_i^{q^h} \\
0  & 0 & \cdots & 0 & \cdots & V_{i-2}^{q^{h+1}} & 0 \\
\vdots &  & & \vdots & & & \vdots \\
0 &  0 & \cdots & 0 & \cdots & 1 & 0
\end{pmatrix}
= (-1)^{h-1}V_i^{q^h}B_{h-i}^{q^i},
\]
so that
\[
Y^{q^h}-Y\equiv (-1)^{h-1}(B_h+V_1^{q^h}B_{h-1}+V_2^{q^h}B_{h-2}+\dots+V_{h-1}^{q^h}B_1) \pmod{\pi^\varepsilon}.
\]
This last expression agrees with the determinant in the proposition, as can be seen by expanding along the first row.
\end{proof}

\subsection{Conclusion of the proof.}\label{conclusion}   We now complete the proof of Thm.~\ref{mainthm}.   Let $x$ be an unramified canonical point on the Lubin-Tate tower.  Since the unramified canonical points in $\mathfrak{X}(1)$ lie in the same orbit under $\OO_B^\times=\Aut \Sigma$, we may assume that $x$ lies above the point with $u_1=\dots=u_{h-1}=0$ in $\mathfrak{X}(1)$.  Recall that $\mathfrak{X}(\pi)^{1,x}\subset \mathfrak{X}(\pi)\otimes E_1$ is the affinoid defined by the conditions $v(u_i)\geq 1$ for $i=1,\dots,h-1$ and $v(X^{(1)}_r-x_r)\geq v(x_r^q)$ for $r=1,\dots,h$;  we showed in Prop.~\ref{trivialization} that $\mathfrak{X}(\pi)^{1,x}$ is a polydisc over $E_1$.

The quotient $\mathfrak{X}(K_{x,2})\to \mathfrak{X}(\pi)$ is Galois with group $H=U_{\mathfrak{A}}^1/K_{x,2}\approx \mathbf{F}_{q^h}$.  After passing to $E_2$ coefficients, the affinoid $\mathfrak{Z}$ was defined as the inverse image of $\mathfrak{X}(\pi)^{1,x}$ in this quotient.  Therefore $\mathfrak{Z}\to \mathfrak{X}(\pi)^{1,x}\otimes E_2$ is an \'etale cover of affinoids with group $H$.  Consider the integral coordinate $Y$ on $\mathfrak{Z}$ produced by Prop.~\ref{Yprop}:  the calculation in~\S\ref{glhf} below shows that the action of a nonzero element of $H$ translates $Y$ by a nonzero element of $\mathbf{F}_{q^h}$.  Thus the reduction of the cover $\mathfrak{Z}\to\mathfrak{X}(\pi)^{1,x}\otimes E_2$ is an \'etale cover of affine hypersurfaces over $\overline{k}$, also with group $H$.

For a tuple $V=(V_1,\dots,V_{h-1})$, let $d(V)$ denote the determinant appearing on the right hand side of Eq.~\eqref{YV}.   Let $\overline{\mathfrak{Z}}'$ denote the hypersurface over $\overline{k}$ with equation $Y^{q^h}-Y=d(V)$;  then $\overline{\mathfrak{Z}}'\to\mathbf{A}^{h-1}$ is an Artin-Schreier cover of affine hypersurfaces with group $H$.  Prop.~\ref{Yprop} shows that $\overline{\mathfrak{Z}}\to\mathbf{A}^{h-1}$ factors through an $H$-equivariant morphism  $\overline{\mathfrak{Z}}\to\overline{\mathfrak{Z}}'$.   Since $\overline{\mathfrak{Z}}$ and $\overline{\mathfrak{Z}}'$ are both \'etale covers of $\mathbf{A}^{h-1}$ with group $H$, we find that $\overline{\mathfrak{Z}}\to\overline{\mathfrak{Z}}'$ is an isomorphism.

Finally, $\overline{\mathfrak{Z}}'$ is isomorphic to the hypersurface described in Thm.~\ref{mainthm} via $Y=(-1)^{h-1}V_h$.  This concludes the proof of Thm.~\ref{mainthm}.

\section{Group actions on a hypersurface}  We close with a discussion of various group actions on the affinoid $\mathfrak{Z}$, with an eye towards linking Thm.~\ref{mainthm} with the local Langlands correspondence and the Jacquet-Langlands correspondence.  What follows is meant to indicate further directions of research;  no proofs will be given.

A large open subgroup of $\GL_h(F)\times B^\times\times W_F$ acts on the Lubin-Tate tower $\mathfrak{X}(\pi^n)$, cf. the introduction to~\cite{HarrisTaylor:LLC}.  To investigate the question of whether the cohomology of the affinoid $\mathfrak{Z}$ realizes the appropriate correspondences among the three factor groups, it will be useful to compute the stabilizer of $\mathfrak{Z}$ in each group, along with the action of the stabilizer on the reduction $\overline{\mathfrak{Z}}$.  We do precisely this for the groups $\GL_h(F)$ and $W_F$.  The hypersurface $\overline{\mathfrak{Z}}$, when considered as an abstract variety over $\overline{k}$, admits a nontrivial action by a large subquotient of $B^\times$, but we cannot prove this action arises from the actual action of $B^\times$ on the Lubin-Tate tower.

Let $X$ be the $\F_{q^h}$-rational model for $\overline{\mathfrak{Z}}/\FF_q$ from Conj.~\ref{mainconj}.  That is, $X\subset \mathbf{A}^h_{\F_{q^h}}$ is the hypersurface with equation \[ \det\left(\begin{matrix}
V_1^{q^h}-V_1 & V_2^{q^h}-V_2 & V_3^{q^h}-V_3 & \cdots & V_{h-1}^{q^h} - V_{h-1} & V_h^{q^h}-V_h
\\
1 & V_1^q & V_2^q & \cdots & V_{h-2}^q & V_{h-1}^q
\\
0 & 1 & V_1^{q^2} & \cdots & V_{h-3}^{q^2} & V_{h-2}^{q^2}
\\
\vdots & & & \ddots & & \vdots
\\
0 & 0 & 0 & \cdots& 1 & V_1^{q^{h-1}} \end{matrix}\right)=0.\]  The actions on $\overline{\mathfrak{Z}}$ we consider in this paragraph all descend to actions on the $\F_{q^h}$-rational model $X$.

\subsection{\texorpdfstring{The action of $\GL_h(F)$}{The action of GLh(F)}}
\label{glhf}

The affioid $\mathfrak{Z}$ is stabilized by the group $U_{\mathfrak{A}}^1=1+\pi M_h(\OO_F)$, and the action of $U_{\mathfrak{A}}^1$ on $\mathfrak{Z}$ factors through the quotient $H=U_{\mathfrak{A}}^1/K_2$.   The action of $H$ on the reduction $\overline{\mathfrak{Z}}$ can be made completely explicit.  We identify $H$ with $k_h=\mathbf{F}_{q^h}$ via the isomorphism $1+\gamma\pi \mapsto \gamma$, $\gamma\in k_h$.   From Eq.~\eqref{gaction} and the construction of $Y$ in Eqs.~\eqref{Yzetadef} and~\eqref{Ydef} we see that the action of an element $\gamma\in H$ on $\overline{\mathfrak{Z}}$ preserves the variables $V_1,\dots,V_{h-1}$ and has the following effect on $V_h$:
\begin{equation}
\label{gamma_action}
V_h\mapsto V_h+\sum_{j=1}^{h-1} \beta^{q^j}\tr_{k_h/k}(\zeta^{q^j}\gamma)=V_h+\gamma.
\end{equation}
Of course, this action descends to an action of $H$ on $X$ by $\F_{q^h}$-rational automorphisms.

We offer some brief remarks relating the characters of the group $H$ to the theory of Bushnell-Kutzko types for $\GL_h(F)$, wherein supercuspidal representations are constructed by induction from compact-mod-center subgroups.  In fact, in our particular situation, the construction goes back to Howe~\cite{Howe}.  Suppose $\psi$ is a character of $H\approx \mathbf{F}_{q^h}$ which does not factor through $\tr_{\F_{q^h}/\F_{q^d}}$ for any proper divisor $d$ of $h$.  This character pulls back to a character of $U_{\mathfrak{A}}^1=1+\pi M_h(\OO_F)$, which we also call $\psi$.  Recall that we have fixed an embedding of $E$ into $M_h(F)$.  Choose a character $\theta$ of $E^\times$ for which $\theta\vert_{1+\gp_E}=\psi\vert_{1+\gp_E}$.  Then $\theta$ is an {\em{admissible}} character in the sense that there is no proper subextension $E'\subset E$ of $E/F$ for which $\theta$ factors through the norm map $E^\times\to (E')^\times$.  The character $\theta$ has conductor $\gp_E^2$.   Let $\eta$ be the unique character of $J=E^\times U^1_{\mathfrak{A}}$ for which $\eta\vert_{E^\times}=\theta$ and $\eta\vert_{U^1_{\mathfrak{A}}}=\psi$.  Then $\pi(\theta)=\Ind_J^{\GL_h(F)}\eta$ is a supercuspidal representation of $\GL_h(F)$ of level $\pi^2$:  This is a special case of the construction used to prove Theorem 2 of~\cite{Howe}.

Therefore the question of whether the cohomology of $\mathfrak{Z}$ realizes the Bushnell-Kutzko types for $\GL_h(F)$ is a matter of determining which characters of $H$ appear in the cohomology of $X$;  this is discussed in Conj.~\ref{mainconjalt} below.

\subsection{The action of inertia}
The action of the inertia subgroup $I_F\subset W_F$ on $\overline{\mathfrak{Z}}$ can be made explicit as well.  Let $I_2=\Gal(E_2/E^{\text{nr}})$;  we identify $I_2$ with $(\OO_E/\pi^2\OO_E)^\times$ via the reciprocity map of local class field theory.  Thus if $\alpha\in\OO_E^\times$, and $x\in\mathfrak{X}(\pi^2)$ is an unramified canonical point corresponding to a basis $x_1,\dots,x_h$ of $\Fc_0[\pi^2]$, then $\alpha(x)$ corresponds to the basis $\alpha x_1,\dots, \alpha x_h$.   Since the definition of the affinoid $\mathfrak{Z}$ only depends on the image of $x$ in $\mathfrak{X}(\pi)$, the stabilizer of $\mathfrak{Z}$ in $I_2$ is the group $(1+\pi\OO_E)/(1+\pi^2\OO_E)$.  The action of an element $1+\gamma\pi \in 1+\pi\OO_E$ on $\overline{\mathfrak{Z}}$ is exactly as in Eq.~\eqref{gamma_action}.

\subsection{\texorpdfstring{The action of $B^\times$}{The action of B*}}
\label{Bcross}

More subtle is the action of $\OO_B^\times = \Aut \Sigma$.  The algebra $\OO_B$ is generated over $\OO_F$ by $\OO_E$ and $\Phi$, where $\Phi^h=\pi$ and $\Phi\alpha=\alpha^q\Phi$, $\alpha\in \OO_E$.   For $n\geq 1$, let $U_B^n=1+\Phi^n\OO_B$.  Let $C^B$ be the orthogonal complement of $\OO_E$ in $\OO_B$, so that \[C^B = \OO_E\Phi\oplus \dots \oplus \OO_E\Phi^{h-1},\]
and define a subgroup $K_2^B$ of $\OO_B^\times$ by
\[ K_2^B = 1+\gp_E^2 + \gp_EC, \]
so that $K_2^B$ lies properly between $U_B^1$ and $U_B^2$.   Let $H^B=U_B^1/K_2^B$.  Let $R$ be the noncommutative ring $\mathbf{F}_{q^h}[\tau]/(\tau^{h+1})$ whose multiplication is given by the rule $\alpha\tau=\tau\alpha^q$, $\alpha\in\mathbf{F}_{q^h}$.  Then $H^B$ is isomorphic to $1+\tau R$.  As we observed in Rmk.~\ref{hypersurfaceremark}, $R^\times$ acts on $X$.  It seems likely that the stabilizer of $\mathfrak{Z}$ in $\OO_B^\times$ is $U_B^1$, and that the action of $U_B^1$ on $\overline{\mathfrak{Z}}$ factors through the this action of $H^B\isom 1+\tau R$.

The action of such a large group of $\F_{q^h}$-rational automorphisms has consequences for the cohomology of $X$ which allow us to reinterpret Conj.~\ref{mainconj}.  First, let us provide a short description of the representation theory of the nilpotent group $H^B$.  The subgroup $Z=1+\tau^h R$ is the center of $1+\tau R\isom H^B$.  Let $\psi$ be a character of $Z\approx \F_{q^h}$ which does not factor through $\tr_{\F_{q^h}/\F_{q^d}}$ for any proper divisor $d$ of $h$.  There is a unique representation $V_\psi$ of $H^B$ lying over $\psi$, of dimension   $q^{h(h-1)/2}$.  Let $\mathcal{H}=H^*_c(X\otimes\FF_q,\QQ_\ell)$, considered as a virtual module for the action of $\Gal(\FF_q/\F_{q^h})\times H^B$.  Conj.~\ref{mainconj} now takes the following alternate form:
\begin{conj}
\label{mainconjalt}   Let $\mathcal{H}_\psi=\Hom_{H^B}\left(V_\psi,\mathcal{H}\right)$, considered as a virtual module for the action of $\Gal(\FF_q/\F_{q^h})$.
Then $\dim\mathcal{H}_\psi=(-1)^{h-1}$, and the eigenvalue of $\Frob_{q^h}$ on $\mathcal{H}_\psi$ is $q^{h(h-1)/2}$.
\end{conj}
The formalism of Bushnell-Kutzko types for $\GL_h(F)$ in~\cite{BushnellKutzko} has been extended to the context of its anisotropic form $B^\times$ by Broussous~\cite{Broussous}.  Granting Conj.~\ref{mainconjalt}, it will not be difficult to detect the types for $B^\times$ in the middle cohomology of $\mathfrak{Z}$.  The types for $B^\times$ appearing in $H_c^{h-1}(\overline{\mathfrak{Z}},\QQ_\ell)$ should correspond exactly to those types for $\GL_h(F)$ which appear there;  indeed this space should realize the correspondence between types.  There has already been much work towards an ``explicit Jacquet-Langlands correspondence", whereby the admissible square-integrable duals of $\GL_h(F)$ and of $B^\times$ are linked via the explicit parameterizations of each dual via types, see~\cite{HenniartJLI},~\cite{HenniartJLII},~\cite{HenniartJLIII}.   However there are still outstanding cases where the explicit Jacquet-Langlands correspondence is not established, including (in some instances) the supercuspidals $\pi(\theta)$ of \S\ref{glhf}. For these there may be some advantage to the cohomological point of view, given that the Jacquet-Langlands correspondence is already known to be realized in the cohomology of the Lubin-Tate tower, cf.~\cite{HarrisTaylor:LLC},~\cite{Strauch}.
%

%

%



\def\cprime{$'$} \def\cprime{$'$}
\providecommand{\bysame}{\leavevmode\hbox to3em{\hrulefill}\thinspace}
\providecommand{\MR}{\relax\ifhmode\unskip\space\fi MR }
\providecommand{\MRhref}[2]{%
  \href{http://www.ams.org/mathscinet-getitem?mr=#1}{#2}
}
\providecommand{\href}[2]{#2}

\end{document}